\documentclass[11pt,reqno]{amsart}

\usepackage{amsfonts,amssymb,amsmath}

\usepackage[cp1251]{inputenc}
\usepackage[T2B]{fontenc}
\usepackage[english]{babel}
\usepackage[mathscr]{eucal}
\usepackage{calrsfs}

\usepackage{tikz}
\usepackage{ragged2e}

\usepackage{cite}  
\usepackage[all]{xy}

\usepackage[mathscr]{eucal}

\usepackage[active]{srcltx}
\allowdisplaybreaks

\theoremstyle{plain}
\newtheorem{thm}{Theorem}[section]
\newtheorem{lem}[thm]{Lemma}
\newtheorem{cor}[thm]{Corollary}

\theoremstyle{definition}
\newtheorem{rem}[thm]{Remark}

\numberwithin{equation}{section}

\newcommand{\sgn}{\mathop{\mathrm{sgn}}}
\def\loc{\text{\rm loc}}
\def\supp{\mathop{\rm supp}\nolimits}

\newcommand{\beqn}{\begin{eqnarray}}
\newcommand{\eeqn}{\end{eqnarray}}
\newcommand{\be}{\begin{equation}}
\newcommand{\ee}{\end{equation}}
\newcommand{\bal}{\begin{align}}
\newcommand{\eal}{\end{align}}

 \def\Mleb{\mathfrak{M}}

\newcommand{\ben}{\begin{enumerate}}
\newcommand{\een}{\end{enumerate}}

\usepackage[left=2.5cm,right=2.5cm,
    top=2.5cm,bottom=2.5cm]{geometry}

\def\wcnn{\mathop{\phantom{W}}\limits^{\circ\circ}\mskip-23muW}
\def\wcn{\mathop{\phantom{W}}\limits^{\circ}\mskip-23muW}

\begin{document}

\begin{center}{\bf ON WEAK ASSOCIATED REFLEXIVITY OF WEIGHTED SOBOLEV SPACES\\ OF THE FIRST ORDER ON REAL LINE}\end{center}

\smallskip\begin{center}
{\bf V.D. Stepanov$^{1,3}$\footnote{Corresponding author: stepanov@mi-ras.ru} and E.P. Ushakova$^{2,3}$}\end{center}

\smallskip
\noindent$^1$\textit{\small Computing Center of Far Eastern Branch of Russian Academy of Sciences, 65 Kim Yu Chena str., Khabarovsk 680000, Russia}

\noindent$^2$\textit{\small V.A. Trapeznikov Institute of Control Sciences of Russian Academy of Sciences, 65 Profsoyuznaya str., Moscow 117997, Russia}

\noindent$^3$\textit{\small Steklov Mathematical Institute of Russian Academy of Sciences, 8 Gubkina str., 119991 Moscow,  Russia}

\medskip


\vskip 0.2cm

{{\bf Abstract.}  
We study associate and double associate spaces of two-weighted Sobolev spaces of the first order on real half-line and we show that unlike the notion of duality the associativity is divided into two cases which we call "strong" and "weak" ones with the division of the second associativity into four cases. On the way we prove that the Sobolev space of compactly supported functions possess weak associated reflexivity and the double weak-strong associate space is vacuous. The case of power weights was recently characterized by reduction to Ces\`{a}ro or Copson type spaces \cite{S1}}.

\smallskip
\section{Introduction}

Let $1<p<\infty, m\in\mathbb{N}$ and let $W^{p,m}, W_0^{p,m}$ and $H^{p,m}$ be classical Sobolev spaces (see \cite[Chapter 3]{A}), where $W_0^{p,m}$ and $H^{p,m}$ are completions of $C^\infty_0$ and $C^m$, respectively, with regard to the norm 
\begin{equation*}
\|f\|_{m,p}:=\left(\sum_{0\leq |\alpha|\leq m}\|D^\alpha f\|_p^p\right)^{\frac{1}{p}}.
\end{equation*}
Moreover, $W^{p,m}=H^{p,m}$ \cite[Theorem 3.16]{A}. If $N=\sum_{0\leq |\alpha|\leq m} 1$ then the dual
of $W^{p,m}$ is  a closed subspace of vector Lebesgue space $L^{p'}_N,$ where $p'=\frac{p}{p-1}.$ It implies reflexivity of $W^{p,m}$ as well as $W^{p,m}_0$ on the base of general criterion of reflexivity of Banach spaces \cite[Theorem 1.17]{A} and weak compactness of a ball in $W^{p,m}$ which follows from \cite[\textsection\, 4, Theorem 2]{S}.  General form of arbitrary linear bounded functional $L\in (W^{p,m})^\prime$ is given by \cite[Theorem 3.8]{A} with implicit formula for the norm $\|L\|.$
Alternatively, $W^{-m,p'}=(W^{p,m}_0)^\prime$ is constructed as completion of the set of functionals $V:=\{L_v; v\in L^{p'}\}\subset (W^{p,m}_0)^\prime,$
$L_v(u):=\langle u,v\rangle:=\int u(x)v(x)dx$ with respect to the norm
\begin{equation}\label{nrm}
\|v\|_{-m,p'}:=\sup_{0\not=u\in W^{p,m}_0}\frac{|\langle u,v\rangle|}{\|u\|_{m,p}}.
\end{equation}
Similar results are known for the Sobolev-Orlicz spaces (see \cite{K} and literature therein).

Generally, elements of $(W^{p,m})^\prime, (W^{p,m}_0)^\prime$ are distributions of positive order.
We learn out the case when duality is replaced by associativity and limit ourselves to the study of the two-weight Sobolev spaces of the first order on the real line. The motivation to characterize associative spaces is that it gives the principle of duality which allows to reduce a problem of the boundedness of a linear operator, say from Sobolev space to Lebesgue space, to a more manageable problem for its conjugate operator (see, for examples \cite{Oin2, Oin3, Oin4, Oin5, Oin6}, \cite{S3}).  

Now we provide basic definitions. Let $I:=(a,b)\subseteq\mathbb{R}$ be an open interval of the real axis and let $\Mleb(I)$ be the set of all Lebesgue measurable functions on $I$ . For $1\le p<\infty$ we denote $L^p(I)\subset \Mleb(I)$ 
the usual Lebesgue space with the norm $\|f\|_{L^p(I)}:=\left(\int_I|f|^p\right)^{1/p}.$ 
Let
$
{\mathscr V}_p(I):=\bigl\{v\in L^p_\loc(I): v\ge 0,\|v\|_{L^1(I)}\not=0 \bigr\}
$
be the set of weight functions (weights) and $v_0,v_1\in {\mathscr V}_1(I)$.
Denote $W^1_{1,\loc}(I)$  the space of all functions $u\in L^1_\loc(I)$, which distributional derivatives $Du$ belong to $L^1_\loc(I)$. We study the weighted Sobolev space
\begin{equation*}
W^1_p(I):=\bigl\{u\in W^1_{1,\loc}(I):\|u\|_{W^1_p(I)}<\infty\bigr\},
\end{equation*}
where
\begin{equation*}
\|u\|_{W^1_p(I)}:=\|v_0 u\|_{L^p(I)}+\|v_1 Du\|_{L^p(I)},
\end{equation*}
and the subspaces $\wcnn^1_p(I)\subset\wcn^1_p(I)\subset W^1_p(I),$ where the second is the closure in  $W^1_p(I)$ of a subspace $\wcnn^1_p(I)$ of all absolutely continuous functions $ AC(I)$ of the form
\begin{equation*}
\wcnn^1_p(I):=\bigl\{f\in AC_\loc(I): f(0)=0,~ \supp f~\text{compact~in~}I,\|f\|_{W^1_p(I)}<\infty\bigr\}.
\end{equation*}
Let $(X,\|\cdot\|_X)$ be a normed space of measurable functions on $I.$ $X$ is called an ideal space provided it satisfies the property: if $|f|\leq|g|$ a.e. on $I$ and $g\in X,$ then $f\in X$ 
and $\|f\|\leq\|g\|.$ Put
\be\label{D-X}
\mathfrak{D}_X:=\Bigl\{g\in \mathfrak{M}(I):\int_I |fg|<\infty\, ~\text{for~ all~}\,f\in X\Bigr\}. 
\ee
For any $g\in\mathfrak{D}_X$ we define the functionals
$$
\mathbf{J}_{X}(g):=\sup_{0\not=f\in X}\frac{\int_I |fg|}{\|f\|_{X}}
\,\, \text{and}\,\,{J}_{X}(g):=\sup_{0\not=f\in X}\frac{|\int_I fg|}{\|f\|_{X}}
 $$
 and the associated spaces
 $$
X'_s:=\bigl\{g\in \Mleb(I):\|g\|_{X'_s}:=\mathbf{J}_X(g)<\infty\bigr\},
$$
$$
X'_w:=\bigl\{g\in \Mleb(I):\|g\|_{X'_w}:={J}_X(g)<\infty\bigr\},
$$
which we call \textquotedblleft strong\textquotedblright\, and \textquotedblleft weak\textquotedblright\, associated spaces, respectively.
A standard problem for an ideal space $(X,\|\cdot\|_X)$ is characterization of the  \textquotedblleft strong\textquotedblright\, associated space (or the K\"{o}the dual) (see \cite[Chapter 1]{BS}).
Observe that $J_X(g)=\mathbf{J}_X(g)$ for an ideal space $X.$ 
For a non-ideal space $J_X(g)$ and $\mathbf{J}_X(g)$ might be different (see \cite{PSU1} for examples). In particular, any weighted Sobolev space
$X\in\{\wcnn^1_p(I),\wcn^1_p(I),W^1_p(I)\}$ is an example for which it might be $J_X(g)\not=\mathbf{J}_X(g)$ \cite{PSU0}, \cite{PSU1}. 

Let $X\in\{\wcnn^1_p(I),\wcn^1_p(I),W^1_p(I)\}.$ A complete characterization of the associate spaces $X'_s$ and $X'_w$ is obtained in \cite[Sections 5, 6]{PSU1}. 
Besides, it was recently discovered that for power weight functions $v_0$ and $v_1$ the spaces $X'_s$ and $X'_w$ coincide with Ces\`{a}ro or Copson type spaces.

It appears a natural problem to characterize "double associate" spaces of the form $[X'_s]'_s,$ $[X'_s]'_w,$ $[X'_w]'_s,$ $[X'_w]'_w.$ Complete analysis of the problem for the Sobolev spaces with power weights and the Ces\`{a}ro or Copson type spaces is given in \cite{P1, S1}.

The main goal of the paper is to establish \textquotedblleft weak\textquotedblright\, associated reflexivity of the Sobolev space $X=\wcnn^1_p(I)$ if $1<p<\infty$. 
For the reflexivity of the \textquotedblleft strong\textquotedblright\, and \textquotedblleft weak\textquotedblright\, weigh\-ted Ces\`{a}ro and Copson type 
spaces see \cite{S2} and \cite{P2}, respectively.

In the next section we provide technical tools to deal with weighted Sobo\-lev spaces and their associated.  In particular, we remind characterization of $X'_s$ 
and $X'_w$ from \cite{PSU0}, when $X=\wcnn^1_p(I)$ and show that $[X'_w]'_s=\{0\}$ in this case (see Corollary \ref{corol}).

The main result is contained in Section 3, where we establish the \textquotedblleft weak\textquotedblright\, as\-so\-ci\-a\-ted reflexivity, 
that is $X=[X'_w]'_w$ of $X=\wcnn^1_p(I)$ (see Theorem \ref{theoremMain}). The characterization of $[X'_s]'_s=[X'_s]'_w$ is still open. However, for the power weights all $[X'_s]'_s,$ $[X'_s]'_w,$ $[X'_w]'_s,$ $[X'_w]'_w$ are discribed \cite{S1}.                                                                                                    

We use signs $:=$ and $=:$ for determining new quantities. We write $A\lesssim B,$ if $A\leq cB$ with some positive constant $c$, which depends only on $p$. $A\approx B$ is equivalent to 
$A\lesssim B \lesssim A$.
Symbols $\mathbb{N}$ and $\mathbb{Z}$ are used for the sets of natural and integer numbers, respectively. Denotation $\chi_E$ means the characteristic function (indicator) of a set $E.$ Uncertainties of 
the form $0\cdot\infty, \frac{\infty}{\infty}$ and $\frac{0}{0}$ are taken to be zero. 
Symbol $\Box$ stands for the end of a proof.  If $1<p<\infty,$ then $p':=\frac{p}{p-1}.$ 

\section{Sobolev spaces and their associated}

Let $1<p<\infty$. 
Suppose for simplicity that $I=(0,\infty)$ and there exists  $c\in (0,\infty)$ for which
\begin{equation}\label{S6}
\|v_1^{-1}\|_{{L^{p'}(0,c)}}\|v_0\|_{{L^p}(0,c)} = \|v_1^{-1}\|_{L^{p'}(c,\infty)}\|v_0\|_{L^p(c,\infty)} = \infty.
\end{equation}
Then by \cite[Lemma 1.6]{Oin} $\wcn^1_p(0,\infty)= W^1_p(0,\infty)$ and by the Oinarov--Otelbaev construction \cite{Oin}, \cite{PSU0}, \cite{PSU1} there exist unique strictly increasing  
absolutely continuous functions $a(t)$  and $b(t)$ such that 
\begin{equation*}
\lim_{t\to 0}a(t)=\lim_{t\to 0}b(t)=0,\qquad\lim_{t\to \infty}a(t)=\lim_{t\to \infty}b(t)=\infty, \qquad
a(t)<t<b(t)\quad (t>0),\end{equation*}
\begin{equation}\label{2}
\int_{a(t)}^t v_1^{-p'}=\int_t^{b(t)}v_1^{-p'},\quad t>0,
\end{equation}
({\sl equilibrium condition}) and
\begin{equation}
\label{3}
\biggl(\int_{a(t)}^{b(t)}v_1^{-p'}\biggr)^{1/p'}
\biggl(\int_{a(t)}^{b(t)}v_0^p\biggr)^{1/p}=1,\quad t>0.
\end{equation}
Put
\begin{equation*}
V_1(t):=\int_{\Delta(t)}v_1^{-p'},\qquad V_1^\pm(t):=\int_{\Delta^\pm(t)}v_1^{-p'}, 
\end{equation*}
\begin{equation*}
\Delta(t):=(a(t),b(t)), ~\Delta^-(t):=(a(t),t),~ \Delta^+(t):=(t,b(t))
\end{equation*}
and let $a^{-1}(t)$ be the function reverse to $a(t)$.
Define 
\begin{gather*}
\mathbb{G}(g):=
\biggl(\int_0^\infty v_1^{-p'}(t)\biggl|\int_t^{a^{-1}(t)}\frac{g(x)}{V_1(x)}\biggl(\int_{a(x)}^tv_1^{-p'} \biggr) dx
\biggr|^{p'}\,dt\biggr)^{1/p'},\\
\mathcal{G}(g):=
\biggl(\int_0^\infty v_1^{-p'}(t)\,V_1^{p'}(t)\, \biggl|\int_t^{a^{-1}(t)}\frac{g(x)}{V_1(x)}
\,dx\biggr|^{p'}\,dt\biggr)^{1/p'},\\
\mathsf{G}(g):=\biggl(\int_0^\infty\biggl(\int_t^{a^{-1}(t)}|g(x)|\,dx\biggr)^{p'}v_1^{-p'}(t) \,dt\biggr)^{1/p'}
\end{gather*} and notate $W_p^1:=W_p^1(0,\infty)$, $\wcn^1_p:=\wcn^1_p(0,\infty)$, $\wcnn_p^1:=\wcnn_p^1(0,\infty)$.
\begin{thm}
\label{T1}
{\rm\cite[Theorem 3.1]{ESU}, \cite[Theorem 4.1]{PSU0}, \cite[Theorem 4.5]{PSU0}}
Let $1<p<\infty$ and $g\in L^1_{\rm loc}(0,\infty)$. Suppose that $v_0,v_1\in {\mathscr V}_p(0,\infty)$, $\frac{1}{v_1}\in L^{p'}_\loc(0,\infty)$ and the condition \eqref{S6} is satisfied.
Then
\begin{align*}
{\mathbf J}_{W_p^1}(g)={\mathbf J}_{\wcnn_p^1}(g)\approx \mathsf{G}(g).
\end{align*}
If $X=W_p^1$ or $X=\wcnn_p^1$, then 
\begin{align*}
{X}^\prime_s=\bigl\{g\in L^1_\loc(0,\infty): \mathsf{G}(g)<\infty, \|g\|_{{X}^\prime_s}\approx \mathsf{G}(g)\bigr\}.
\end{align*}
Secondly,
\begin{equation}\label{UB}
J_{\wcnn_p^1}(g)\approx \mathbb{G}(g)+\mathcal{G}(g),
\end{equation}
and if $X=\wcnn_p^1$, then 
\begin{align*}\label{ESU12}
X^\prime_w=\bigl\{g\in L^1_\loc(0,\infty): \mathbb{G}(g)+\mathcal{G}(g)<\infty, \|g\|_{X^\prime_w}\approx \mathbb{G}(g)+\mathcal{G}(g)\bigr\}.
\end{align*}
Also, $J_{W_p^1}(g)<\infty$ if and only if $\mathsf{G}(g)<\infty$ and  $J_{W_p^1}(g)\approx \mathbb{G}(g)+\mathcal{G}(g)$.
\end{thm}

\begin{rem}\label{rm}
Let $v_0=v_1\equiv 1.$ Then we can open the right hand side of \eqref{nrm} for $W^{1,p}(0,\infty),$ using \cite[Example 7.2]{PSU1}. Namely, we have
\begin{align*}
\|v\|_{-1,p'}\approx&
\Biggl(\int_0^\infty \biggl|\int_t^{t+\frac{1}{2}}v
\biggr|^{p'}\, dt\Biggr)^{\frac{1}{p'}}\\
& +
\Biggl(\int_0^{\frac{1}{2}} t^{-p'}\biggl|\int_0^t\left(\int_t^{y+\frac{1}{2}}v\right)\,dy
\biggr|^{p'}\, dt
+\int_{\frac{1}{2}}^\infty \biggl|\int_{t-\frac{1}{2}}^t\left(\int_t^{y+\frac{1}{2}}v\right)\,dy
\biggr|^{p'}\, dt\Biggr)^{\frac{1}{p'}}.
\end{align*}
\end{rem}

\begin{lem}\label{Norma}
Let $1<p<\infty$ and $X=\wcnn_p^1$, then the functional $\|g\|_{X^\prime_w}$ is a norm. 
\end{lem}
\begin{proof}
It is sufficient to show that 
$$
\|g\|_{X^\prime_w}=0\ \ \Rightarrow\ \ g=0\ \text{a.e. on}\ (0,\infty).
$$
Let $\|g\|_{X^\prime_w}=0$. Then $\mathbb{G}(g)=\mathcal{G}(g)=0$. In particular,
$$
G(t):=\int_t^{a^{-1}(t)}\frac{g(x)}{V_1(x)}\biggl(\int_{a(x)}^tv_1^{-p'} \biggr)\, dx=0\ \ \text{a.e. on}\ (0,\infty).
$$
Hence,
$$
0=G^\prime(t)=\frac{g(t)}{2}\ \ \text{a.e. on}\ (0,\infty).
$$ 
\end{proof} 
Let $1<r<\infty,$ $u\in {\mathscr V}_r(0,\infty).$ Denote 
\begin{gather*}
L^r_u(0,\infty):=\bigl\{h: \|h\|_{r,u}:=\|uh\|_{L^r(0,\infty)}<\infty\bigr\},\\
\mathbb{W}_{p',1/{v_1}}:=\bigl\{g\in L^1_\loc(0,\infty): \|g\|_{\mathbb{W}_{p',1/{v_1}}}:= \mathsf{G}(g)<\infty\bigr\},\\
\mathscr{W}_{p',1/{v_1}}:=\bigl\{g\in L^1_\loc(0,\infty):  \|g\|_{\mathscr{W}_{p',1/{v_1}}}:=\mathbb{G}(g)+\mathcal{G}(g)<\infty\bigr\}.
\end{gather*}

\begin{rem}\label{remark}
From \eqref{UB} we obtain H\"{o}lder's type inequality (see \cite[Theorem 2.4]{BS}) in $\wcnn_p^1$ and $\mathscr{W}_{p',1/{v_1}}$: if $1<p<\infty$ then
$$
\biggl|\int_0^\infty fg\biggr|\lesssim \|f\|_{\wcnn_p^1}\|g\|_{\mathscr{W}_{p',1/{v_1}}}\quad\textrm{for any  }f\in \wcnn_p^1\textrm{ and  }g\in \mathscr{W}_{p',1/{v_1}}.
$$ 
\end{rem}

The norm in $\mathscr{W}_{p',1/{v_1}}$ admits an alternative formulation in terms of a sequence $\{\eta_k\}_{k\in\mathbb{Z}}$ of the form:
$$\eta_0=1, \qquad \eta_{k}=a^{-1}(\eta_{k-1})\quad (k\in\mathbb{N}), \qquad \eta_{k}=a(\eta_{k+1})\quad (-k\in\mathbb{N}).$$ To be able to declare it in the next lemma we denote \begin{equation*}
G^{(\delta)}(t):={[V_1(t)]^{\delta}}\int_t^{a^{-1}(t)}
\frac{g(x)}{V_1(x)}\Bigl(
\int_{a(x)}^tv_1^{-p'}\Bigr)^{1-\delta} dx,\ \ \ \delta=0,1,
\end{equation*} 
and observe that for $t\in[\eta_{k-1},\eta_k]$
\begin{gather}
G^{(\delta)}(t)=G_{1,k}^{(\delta)}(t)+G_{2,k}^{(\delta)}(t),\label{22_1}\\ 
G_{1,k}^{(\delta)}(t):={V^{\delta}_1(t)}\int_t^{\eta_k}
\frac{g(x)}{V_1(x)}\Bigl(
\int_{a(x)}^tv_1^{-p'}\Bigr)^{1-\delta} dx,\nonumber\\
G_{2,k}^{(\delta)}(t):={V^{\delta}_1(t)}\int_{\eta_k}^{a^{-1}(t)}
\frac{g(x)}{V_1(x)}\Bigl(
\int_{a(x)}^tv_1^{-p'}\Bigr)^{1-\delta} dx.\nonumber
\end{gather}

\begin{lem}\label{norm}Let $1<p<\infty,$  $v_0,v_1\in {\mathscr V}_p(0,\infty)$, $\frac{1}{v_1}\in L^{p'}_\loc(0,\infty)$ and the condition \eqref{S6} is satisfied.
Then \begin{align}
\|g\|_{\mathscr{W}_{p',1/{v_1}}}^{p'}\approx& \sum_{k\in\mathbb{Z}}\biggl\{\int_{\eta_{k-1}}^{\eta_k}
v_1^{-p'}(t)\bigl|G_{1,k}^{(0)}(t)\bigr|^{p'}\,dt+\int_{\eta_{k-1}}^{\eta_k}
v_1^{-p'}(t)\bigl|G_{2,k}^{(0)}(t)
\bigr|^{p'}\,dt \nonumber \\
&+\int_{\eta_{k-1}}^{\eta_k}
v_1^{-p'}(t)\bigl|G_{1,k}^{(1)}(t)\bigr|^{p'}\,dt+\int_{\eta_{k-1}}^{\eta_k}
v_1^{-p'}(t)\bigl|G_{2,k}^{(1)}(t)
\bigr|^{p'}\,dt\biggr\}.\label{qq1}
\end{align} 
\end{lem}

\begin{proof} 
The upper estimate follows from \eqref{22_1} and
$$
\|g\|_{\mathscr{W}_{p',1/{v_1}}}^{p'}\lesssim\sum_{k\in\mathbb{Z}}\int_{\eta_{k-1}}^{\eta_k}v_1^{-p'}(t)\Bigr\{\bigl|
G^{(0)}(t)\bigr|^{p'}+\bigl|
G^{(1)}(t)\bigr|^{p'}\Bigr\}dt.
$$

To establish the lower estimate we assume that the inequality
$$
\biggl|\int_0^\infty fg\biggr|\le C\|f\|_{\wcnn^1_p}=
C\Bigl\{\|fv_0\|_p+\|f'v_1\|_p\Bigr\}
$$ 
holds with $C=\|g\|_{\mathscr{W}_{p',1/{v_1}}}$, and let for some $N\in\mathbb{N}$
\begin{align*}
{F}_{1,N}^{(\delta)}(x):=
\frac{\sum_{|k|\le N}\chi_{[\eta_{k-1},\eta_k]}(x)}{V_1^-(x)}
\int_{\eta_{k-1}}^x v_1^{-p'}(t)\bigl[\sgn G^{(\delta)}_{1,k}(t)\bigr] \biggl(\int_{a(x)}^t v_1^{-p'}\biggr)^{1-\delta}
[V_1(t)]^{\delta}\bigl|G^{(\delta)}_{1,k}(t)\bigr|^{p'-1}\,dt,\end{align*}
\begin{align*}
{F}_{2,N}^{(\delta)}(x):=
\frac{\sum_{|k|\le N}\chi_{[\eta_{k},\eta_{k+1}]}(x)}{V_1^-(x)}
\int_{a(x)}^{\eta_{k}} v_1^{-p'}(t)\bigl[\sgn G^{(\delta)}_{2,k}(t)\bigr]\biggl(\int_{a(x)}^t v_1^{-p'}\biggr)^{1-\delta}
[V_1(t)]^{\delta}\bigl|G^{(\delta)}_{2,k}(t)\bigr|^{p'-1}\,dt.
\end{align*} 
If $f={F}^{(\delta)}_{1,N}+{F}^{(\delta)}_{2,N}$ then
\begin{equation}\label{0}\int_0^\infty g(x)f(x)\,dx=
\sum_{|k|\le N}\biggl\{\int_{\eta_{k-1}}^{\eta_k}
v_1^{-p'}(t)\bigl|G_{1,k}^{(\delta)}(t)\bigr|^{p'}\,dt+\int_{\eta_{k-1}}^{\eta_k}
v_1^{-p'}(t)\bigl|G_{2,k}^{(\delta)}(t)
\bigr|^{p'}\,dt\biggr\}. 
\end{equation} 
To evaluate
\begin{align*}
\|{F}^{(\delta)}_{1,N}v_0\|_p^p=&\sum_{|k|\le N}\int_{\eta_{k-1}}^{\eta_k} v_0^{p}(x)
\biggl|\frac{1}{V_1^-(x)}\int_{\eta_{k-1}}^x v_1^{-p'}(t)
\bigl[\sgn G^{(\delta)}_{1,k}(t)\bigr]\\&\times \biggl(\int_{a(x)}^t v_1^{-p'}\biggr)^{1-\delta}
[V_1(t)]^{\delta}\bigl|G^{(\delta)}_{1,k}(t)\bigr|^{p'-1}\,dt\biggr|^p\,dx
\end{align*} 
we apply well known characterization of weighted Hardy's inequality \cite[p. 6]{KPS}, in order to obtain
\begin{equation*}
\int_{\eta_{k-1}}^{\eta_k} v_0^{p}(x)
\biggl(\int_{\eta_{k-1}}^x v_1^{-p'}(t)
\bigl|G^{(\delta)}_{1,k}(t)\bigr|^{p'-1}\,dt\biggr)^p\,dx
\lesssim A_{1}^p\int_{\eta_{k-1}}^{\eta_k}
v_1^{-p'}\bigl|G_{1,k}^{(\delta)}\bigr|^{p'},
\end{equation*} 
where (see \eqref{3})
\begin{align*}
A_{1}:=\sup_{\eta_{k-1}<t<\eta_k}\biggl(\int_t^{\eta_k}  v_0^{p}\biggr)^{1/p}\biggl(\int_{\eta_{k-1}}^t v_1^{-p'}\biggr)^{1/p'}\le
\biggl(\int_{\eta_{k-1}}^{\eta_k}  v_0^{p}\biggr)^{1/p}\biggl(\int_{\eta_{k-1}}^{\eta_{k}} v_1^{-p'}\biggr)^{1/p'}\le  1.
\end{align*} 
Therefore, by using in the $\delta=1\,-$case the relation 
\begin{equation}\label{Gk1} V_1(t)=2V_1^+(t)\le 2\int_{\eta_{k-1}}^{b(t)}v_1^{-p'}\le 2\int_{\eta_{k-1}}^{b(x)}v_1^{-p'} \le 2V_1(x)=4V_1^-(x), \quad \eta_{k-1}\le t\le x,\end{equation} we have for the both $\delta=0,1$: 
\begin{align}\label{1}
\|{F}^{(\delta)}_{1,N}v_0\|_p^p\le&
\sum_{|k|\le N}\int_{\eta_{k-1}}^{\eta_k} v_0^{p}(x)
\biggl(\int_{\eta_{k-1}}^x v_1^{-p'}(t)
\bigl|G^{(\delta)}_{1,k}(t)\bigr|^{p'-1}\,dt\biggr)^p\,dx\nonumber\\\lesssim&
\sum_{|k|\le N}\int_{\eta_{k-1}}^{\eta_k}
v_1^{-p'}\bigl|G_{1,k}^{(\delta)}\bigr|^{p'}
=:\bigl[\mathbf{G}^{(\delta)}_{1,N}(g)\bigr]^{p'}.
\end{align}
Analogously, we evaluate, by making use of \begin{equation}\label{Gk2}V_1(t)=2V_1^+(t)\le 2\int_{a(x)}^{b(\eta_k)}v_1^{-p'}\le 2V_1(x)=4 V_1^-(x), \quad \eta_k\le x\le\eta_{k+1},\end{equation} that 
\begin{align*}
\|{F}^{(\delta)}_{2,N}v_0\|_p^p=&\sum_{|k|\le N}\int_{\eta_{k}}^{\eta_{k+1}} v_0^{p}(x)
\biggl|\frac{1}{V_1^-(x)}\int_{a(x)}^{\eta_{k}} v_1^{-p'}(t)
\bigl[\sgn G^{(\delta)}_{2,k}(t)\bigr]\\&\times \biggl(\int_{a(x)}^t v_1^{-p'}\biggr)^{1-\delta}
[V_1(t)]^{\delta}\bigl|G^{(\delta)}_{2,k}(t)\bigr|^{p'-1}\,dt\biggr|^p\,dx\\
\le& \sum_{|k|\le N}\int_{\eta_{k}}^{\eta_{k+1}} v_0^{p}(x)
\biggl(\int_{a(x)}^{\eta_{k}} v_1^{-p'}(t)
\bigl|G^{(\delta)}_{2,k}(t)\bigr|^{p'-1}\,dt\biggr)^p\,dx 
\lesssim \, A_{2}^p\int_{\eta_{k-1}}^{\eta_k}
v_1^{-p'}\bigl|G_{1,k}^{(\delta)}\bigr|^{p'},
\end{align*} 
where (see \eqref{3})
\begin{equation*}
A_{2}:=\sup_{\eta_{k-1}<t<\eta_k}\biggl(\int_{\eta_k}^{a^{-1}(t)}  v_0^{p}\biggr)^{1/p}\biggl(\int_t^{\eta_{k}} v_1^{-p'}\biggr)^{1/p'}\le 1.
\end{equation*} 
Therefore,
\begin{equation}\label{11}
\|{F}^{(\delta)}_{2,N}v_0\|_p^p
\lesssim
\sum_{|k|\le N}\int_{\eta_{k-1}}^{\eta_k}
v_1^{-p'}\bigl|G_{2,k}^{(\delta)}\bigr|^{p'}=:\bigl[\mathbf{G}^{(\delta)}_{2,N}(g)\bigr]^{p'}.
\end{equation}
Further, since 
\begin{align*}
[{F}_{1,N}^{(\delta)}(x)]'=&-\sum_{|k|\le N}\chi_{[\eta_{k-1},\eta_k]}(x)
\frac{\bigl[V_1^-(x)\bigr]'}{\bigl[V_1^-(x)\bigr]^2}\int_{\eta_{k-1}}^x v_1^{-p'}(t)\bigl[\sgn G^{(\delta)}_{1,k}(t)\bigr]\\&\times 
\biggl(\int_{a(x)}^t v_1^{-p'}\biggr)^{1-\delta}
[V_1(t)]^{\delta}\bigl|G^{(\delta)}_{1,k}(t)\bigr|^{p'-1}\,dt+\sum_{|k|\le N}\chi_{[\eta_{k-1},\eta_k]}(x)\\&\times\begin{cases}
v_1^{-p'}(x)\bigl[\sgn G^{(0)}_{1,k}(x)\bigr] \bigl|G^{(0)}_{1,k}(x)\bigr|^{p'-1}\\
-\displaystyle\frac{v_1^{-p'}(a(x))\,a'(x)}{V_1^-(x)}\int_{\eta_{k-1}}^x v_1^{-p'}\,\bigl[\sgn G^{(0)}_{1,k}\bigr] \bigl|G^{(0)}_{1,k}\bigr|^{p'-1}, & \delta=0,\\
2v_1^{-p'}(x)\bigl[\sgn G^{(1)}_{1,k}(x)\bigr] \bigl|G^{(1)}_{1,k}(x)\bigr|^{p'-1}, & \delta=1,\end{cases}
\end{align*}
\begin{align*}
[{F}_{2,N}^{(\delta)}(x)]'=&-\sum_{|k|\le N}\chi_{[\eta_{k},\eta_{k+1}]}(x)
\frac{\bigl[V_1^-(x)\bigr]'}{\bigl[V_1^-(x)\bigr]^2}\int_{a(x)}^{\eta_{k}} v_1^{-p'}(t)\bigl[\sgn G^{(\delta)}_{2,k}(t)\bigr] \\&\times\biggl(\int_{a(x)}^t v_1^{-p'}\biggr)^{1-\delta}
[V_1(t)]^{\delta}\bigl|G^{(\delta)}_{2,k}(t)\bigr|^{p'-1}\,dt-\sum_{|k|\le N}\chi_{[\eta_{k},\eta_{k+1}]}(x)\\&\times\begin{cases}
\displaystyle\frac{v_1^{-p'}(a(x))\,a'(x)}{V_1^-(x)}\int_{a(x)}^{\eta_{k}} v_1^{-p'}\,\bigl[\sgn G^{(0)}_{2,k}\bigr] \bigl|G^{(0)}_{2,k}\bigr|^{p'-1}, & \delta=0,\\
\displaystyle\frac{v_1^{-p'}(a(x))\,a'(x)}{V_1^-(x)}&\\
\times\bigl[\sgn G^{(1)}_{2,k}(a(x))\bigr]V_1^-(a(x)) \bigl|G^{(1)}_{2,k}(a(x))\bigr|^{p'-1}, & \delta=1,\end{cases}
\end{align*} 
then
\begin{equation*}\label{8}
\|[{F}_{1,N}^{(\delta)}]'v_1\|_p
\le\begin{cases} I_1+\bigl[\mathbf{G}^{(0)}_{1,N}(g)\bigr]^{p'-1}+II_1, &\delta=0,\\
I_1+\bigl[\mathbf{G}^{(1)}_{1,N}(g)\bigr]^{p'-1}, &\delta=1,\end{cases}
\end{equation*} 
where
\begin{align*}
I_1^p:=\sum_{|k|\le N}\int_{\eta_{k-1}}^{\eta_k} v_1^{p}(x)
\frac{\Bigl|\bigl[V_1^-(x)\bigr]'\Bigr|^p}{\bigl[V_1^-(x)\bigr]^{2p}} \biggl(\int_{\eta_{k-1}}^x v_1^{-p'}(t) \biggl(\int_{a(x)}^t v_1^{-p'}\biggr)^{1-\delta}
[V_1(t)]^{\delta}\bigl|G^{(\delta)}_{1,k}(t)\bigr|^{p'-1}\,dt\biggr)^p\,dx
\end{align*} 
and 
\begin{align*}
II_1^p:=\sum_{|k|\le N}\int_{\eta_{k-1}}^{\eta_k} v_1^{p}(x)\bigl[V_1^-(x)\bigr]^{-p}\bigl[v_1^{-p'}(a(x))\,a'(x)\bigr]^p
 \biggl(\int_{\eta_{k-1}}^x v_1^{-p'}(t) \bigl|G^{(0)}_{1,k}(t)\bigr|^{p'-1}\,dt\biggr)^p\,dx.
\end{align*}
In view of $v_1^{-p'}(a(x))a'(x)\le 2v_1^{-p'}(x)$ (see \eqref{eq}), we obtain, by using \eqref{Gk1} in the $\delta=1\,-$case, that
\begin{align*}
I_1^p\le&\sum_{|k|\le N}\int_{\eta_{k-1}}^{\eta_k} v_1^{p}(x)
\frac{\bigl|v_1^{-p'}(x)-v_1^{-p'}(a(x))a'(x)\bigr|^p}{\bigl[V_1^-(x)\bigr]^{p}} \biggl(\int_{\eta_{k-1}}^xv_1^{-p'}\bigl|G^{(\delta)}_{1,k}\bigr|^{p'-1}\biggr)^p\,dx\\
\le& \sum_{|k|\le N}\int_{\eta_{k-1}}^{\eta_k} v_1^{p}(x)
\frac{\bigl[v_1^{-p'}(x)+v_1^{-p'}(a(x))a'(x)\bigr]^p}{\bigl[V_1^-(x)\bigr]^{p}} \biggl(\int_{\eta_{k-1}}^xv_1^{-p'}\bigl|G^{(\delta)}_{1,k}\bigr|^{p'-1}\biggr)^p\,dx\\
\le& 3^{p}\sum_{|k|\le N}\int_{\eta_{k-1}}^{\eta_k} v_1^{-p'}(x)
\bigl[V_1^-(x)\bigr]^{-p} \biggl(\int_{\eta_{k-1}}^xv_1^{-p'}\bigl|G^{(\delta)}_{1,k}\bigr|^{p'-1}\biggr)^p\,dx.
\end{align*} 
Analogously, 
\begin{equation*}
II_2^p\le 2^p\sum_{|k|\le N}\int_{\eta_{k-1}}^{\eta_k} v_1^{-p'}(x)\bigl[V_1^-(x)\bigr]^{-p} \biggl(\int_{\eta_{k-1}}^xv_1^{-p'}\bigl|G^{(0)}_{1,k}\bigr|^{p'-1}\biggr)^p\,dx.
\end{equation*} On the strength of the boundedness characteristics for the Hardy operator \cite[p. 6]{KPS},
\begin{align*}
\int_{\eta_{k-1}}^{\eta_k} v_1^{-p'}(x)
\bigl[V_1^-(x)\bigr]^{-p} \biggl(\int_{\eta_{k-1}}^xv_1^{-p'}(t)\bigl|G^{(\delta)}_{1,k}(t)\bigr|^{p'-1}\,dt\biggr)^p\,dx\lesssim \mathbb{A}_1^p\int_{\eta_{k-1}}^{\eta_k}
v_1^{-p'}\bigl|G_{1,k}^{(\delta)}\bigr|^{p'},
\end{align*} 
where 
\begin{equation*}
\mathbb{A}_1:=\sup_{\eta_{k-1}<t<\eta_k}
\biggl(\int_t^{\eta_{k}}  v_1^{-p'}(x)
\bigl[V_1^-(x)\bigr]^{-p}\,dx\biggr)^{1/p}\biggl(\int_{\eta_{k-1}}^t v_1^{-p'}\biggr)^{1/p'}.
\end{equation*} 
It holds
\begin{align*}\label{A1}
\mathbb{A}_1^p\le&\sup_{\eta_{k-1}<t<\eta_k}
\biggl(\int_t^{\eta_{k}}  v_1^{-p'}(x)\biggl(\int_{\eta_{k-1}}^x v_1^{-p'}\biggr)^{-p}\,dx\biggr)\biggl(\int_{\eta_{k-1}}^t v_1^{-p'}\biggr)^{p-1}\\
=&\frac{1}{p-1}\sup_{\eta_{k-1}<t<\eta_k}
\biggl[\biggl(\int_{\eta_{k-1}}^t v_1^{-p'}\biggr)^{1-p}- \biggl(\int_{\eta_{k-1}}^{\eta_k} v_1^{-p'}\biggr)^{1-p}\biggr]
\biggl(\int_{\eta_{k-1}}^t v_1^{-p'}\biggr)^{p-1}
\le\frac{1}{p-1}.
\end{align*} 
Therefore,
\begin{align*}
\sum_{|k|\le N}\int_{\eta_{k-1}}^{\eta_k} \frac{v_1^{-p'}(x)}{
\bigl[V_1^-(x)\bigr]^{p}} \biggl(\int_{\eta_{k-1}}^xv_1^{-p'}(t)\bigl|G^{(\delta)}_{1,k}(t)\bigr|^{p'-1}\,dt\biggr)^p\,dx\lesssim\int_{\eta_{k-1}}^{\eta_k}
v_1^{-p'}\bigl|G_{1,k}^{(\delta)}\bigr|^{p'}=
\bigl[\mathbf{G}^{(\delta)}_{1,N}(g)\bigr]^{p'-1},
\end{align*} 
that is
\begin{equation*}\label{8''}
\|[{F}_{1,N}^{(\delta)}]'v_1\|_p
\lesssim
\bigl[\mathbf{G}^{(\delta)}_{1,N}(g)\bigr]^{p'-1},
\end{equation*} 
and, by letting $N\to\infty$, the estimate
\begin{align}\label{FH}
\|g\|_{\mathscr{W}_{p',1/{v_1}}}^{p'}\gtrsim 
\sum_{k\in\mathbb{Z}}\biggl\{\int_{\eta_{k-1}}^{\eta_k}
v_1^{-p'}(t)\bigl|G_{1,k}^{(0)}(t)\bigr|^{p'}\,dt+\int_{\eta_{k-1}}^{\eta_k}
v_1^{-p'}(t)\bigl|G_{1,k}^{(1)}(t)\bigr|^{p'}\,dt\biggr\}
\end{align}
 is now performed, basing on \eqref{1} and \eqref{0}.
 
Similarly, in view of $V_1^-(x)=\frac{1}{2}V_1(x)\ge 
\frac{1}{2}V_1^+(a(x))=\frac{1}{4}V_1(a(x))$ (for $\delta=1$),
\begin{equation*}\label{8'}
\|[{F}_{2,N}^{(\delta)}]'v_1\|_p
\le\begin{cases} I_2+II_2, &\delta=0,\\
I_2+\bigl[{G}^{(2)}_{1,N}(g)\bigr]^{p'-1}, &\delta=1,\end{cases}
\end{equation*} 
where
\begin{align*}
I_2^p:=\sum_{|k|\le N}\int_{\eta_{k}}^{\eta_{k+1}} v_1^{p}(x)
\frac{\Bigl|\bigl[V_1^-(x)\bigr]'\Bigr|^p}{\bigl[V_1^-(x)\bigr]^{2p}}
\biggl(\int_{a(x)}^{\eta_{k}} v_1^{-p'}(t)\biggl(\int_{a(x)}^t v_1^{-p'}\biggr)^{1-\delta}
[V_1(t)]^{\delta}\bigl|G^{(\delta)}_{2,k}(t)\bigr|^{p'-1}\,dt\biggr)^p\,dx
\end{align*} 
and 
\begin{equation*}
II_2^p:=\sum_{|k|\le N}\int_{\eta_{k}}^{\eta_{k+1}} \frac{v_1^{p}(x)}{\bigl[V_1^-(x)\bigr]^{p}}\bigl[v_1^{-p'}(a(x))\,a'(x)\bigr]^p \biggl(\int_{a(x)}^{\eta_{k}} 
v_1^{-p'}\bigl|G^{(0)}_{2,k}\bigr|^{p'-1}\biggr)^p\,dx,
\end{equation*} 
we obtain analogously to the previous case (see also \eqref{Gk2} for $\delta=1$):
\begin{align*}
I_2^p\le&\sum_{|k|\le N}\int_{\eta_{k}}^{\eta_{k+1}} v_1^{p}(x)
\frac{\bigl|v_1^{-p'}(x)-v_1^{-p'}(a(x))a'(x)\bigr|^p}{\bigl[V_1^-(x)\bigr]^{p}} \biggl(\int_{a(x)}^{\eta_{k}} v_1^{-p'}\bigl|G^{(\delta)}_{2,k}\bigr|^{p'-1}\biggr)^p\,dx\\
\lesssim& \sum_{|k|\le N}\int_{\eta_{k}}^{\eta_{k+1}} v_1^{-p'}(x)
\bigl[V_1^-(x)\bigr]^{-p} \biggl(\int_{a(x)}^{\eta_{k}} v_1^{-p'}\bigl|G^{(\delta)}_{2,k}\bigr|^{p'-1}\biggr)^p\,dx
\end{align*} 
and 
\begin{equation*}
II_2^p\lesssim\sum_{|k|\le N}\int_{\eta_{k}}^{\eta_{k+1}} v_1^{-p'}(x)\bigl[V_1^-(x)\bigr]^{-p} \biggl(\int_{a(x)}^{\eta_{k}} v_1^{-p'}(t)\bigl|G^{(0)}_{2,k}(t)\bigr|^{p'-1}\,dt\biggr)^p\,dx.
\end{equation*} 
By characteristics for the Hardy inequality \cite[p. 6]{KPS},
\begin{align*}
\int_{\eta_{k}}^{\eta_{k+1}} v_1^{-p'}(x)\bigl[V_1^-(x)\bigr]^{-p} \biggl(\int_{a(x)}^{\eta_{k}} v_1^{-p'}(t)\bigl|G^{(0)}_{2,k}(t)\bigr|^{p'-1}\,dt\biggr)^p\,dx
\lesssim \mathbb{A}_2^p\int_{\eta_{k-1}}^{\eta_k}
v_1^{-p'}\bigl|G_{2,k}^{(\delta)}\bigr|^{p'},
\end{align*} 
where 
\begin{equation*}
\mathbb{A}_2:=\sup_{\eta_{k-1}<t<\eta_k}
\biggl(\int_{\eta_{k}}^{a^{-1}(t)}  v_1^{-p'}(x)
\bigl[V_1^-(x)\bigr]^{-p}\,dx\biggr)^{1/p}\biggl(\int_t^{\eta_{k}} v_1^{-p'}\biggr)^{1/p'}.
\end{equation*} 
We have
\begin{align*}
\mathbb{A}_2^p\le&\sup_{\eta_{k-1}<t<\eta_k}
\biggl(\int_{\eta_{k}}^{a^{-1}(t)}  v_1^{-p'}(x)\biggl(\int_{t}^x v_1^{-p'}\biggr)^{-p}\,dx\biggr)\biggl(\int_t^{\eta_{k}} v_1^{-p'}\biggr)^{p-1}\\
=&\frac{1}{p-1}\sup_{\eta_{k-1}<t<\eta_k}
\biggl[\biggl(\int_t^{\eta_{k}} v_1^{-p'}\biggr)^{1-p}- \biggl(\int_t^{a^{-1}(t)} v_1^{-p'}\biggr)^{1-p}\biggr]
\biggl(\int_t^{\eta_{k}} v_1^{-p'}\biggr)^{p-1}
\le\frac{1}{p-1}.
\end{align*} 
Therefore,
\begin{align*}
\sum_{|k|\le N}\int_{\eta_{k}}^{\eta_{k+1}} \frac{v_1^{-p'}(x)}{
\bigl[V_1^-(x)\bigr]^{p}} \biggl(\int_{a(x)}^{\eta_{k}}v_1^{-p'}(t)\bigl|G^{(\delta)}_{2,k}(t)\bigr|^{p'-1}\,dt\biggr)^p\,dx\lesssim\int_{\eta_{k-1}}^{\eta_k}
v_1^{-p'}\bigl|G_{2,k}^{(\delta)}\bigr|^{p'}=
\bigl[\mathbf{G}^{(\delta)}_{2,N}(g)\bigr]^{p'-1},
\end{align*} 
which, in combination with \eqref{11} and \eqref{0}, yields the estimate 
$$
\|g\|_{\mathscr{W}_{p',1/{v_1}}}^{p'}\gtrsim 
\sum_{k\in\mathbb{Z}}\biggl\{\int_{\eta_{k-1}}^{\eta_k}
v_1^{-p'}(t)\bigl|G_{2,k}^{(0)}(t)\bigr|^{p'}\,dt+\int_{\eta_{k-1}}^{\eta_k}
v_1^{-p'}(t)\bigl|G_{2,k}^{(1)}(t)\bigr|^{p'}\,dt\biggr\},
$$ 
by letting $N\to\infty$. Thus, (see also \eqref{FH}) the required lower bound is now confirmed.
\end{proof}

Basing on Lemma \ref{norm} one can prove the following 

\begin{lem}\label{plot}Let $1<p<\infty,$  $v_0,v_1\in {\mathscr V}_p(0,\infty)$, $\frac{1}{v_1}\in L^{p'}_\loc(0,\infty)$ and the condition \eqref{S6} is satisfied.
Then the space $\mathbb{W}_{p',1/{v_1}}$ is dense in $\mathscr{W}_{p',1/{v_1}}.$
\end{lem}

\begin{proof} 
Let $g\in \mathscr{W}_{p',1/{v_1}}$. Then $\|g\|_{\mathscr{W}_{p',1/{v_1}}}<\infty$ by \eqref{qq1}.
Therefore,
\begin{align}\label{r1}
\lim_{n\to\infty}\sum_{|k|\ge n}\biggl\{\int_{\eta_{k-1}}^{\eta_k}
v_1^{-p'}(t)\bigl|G_{1,k}^{(0)}(t)\bigr|^{p'}\,dt+\int_{\eta_{k-1}}^{\eta_k}
v_1^{-p'}(t)\bigl|G_{2,k}^{(0)}(t)
\bigr|^{p'}\,dt\nonumber\\
+\int_{\eta_{k-1}}^{\eta_k}
v_1^{-p'}(t)\bigl|G_{1,k}^{(1)}(t)\bigr|^{p'}\,dt+\int_{\eta_{k-1}}^{\eta_k}
v_1^{-p'}(t)\bigl|G_{2,k}^{(1)}(t)
\bigr|^{p'}\,dt\biggr\}=0.
\end{align}

Let $g_N:=\chi_{[\eta_{-N},\eta_N]}g$ with some $N\in\mathbb{N}$. Then $g_N\in \mathbb{W}_{p',1/{v_1}}$. Indeed, 
\begin{equation*}
G(|g_N|)^{p'}=\biggl\{\int_0^{\eta_{-N-1}}+\int_{\eta_{-N-1}}^{\eta_{N}}+\int_{\xi_{N}}^{\infty}\biggr\}v_1^{-p'}(x)\biggl(\int_x^{a^{-1}(x)}|g_N|\biggr)^{p'}dx, 
\end{equation*} 
where
\begin{align*}
\int_0^{\eta_{-N-1}}v_1^{-p'}(x)\biggl(\int_x^{a^{-1}(x)}|\chi_{[\eta_{-N},\eta_N]}g|\biggr)^{p'}dx=0=
\int_{\eta_{N}}^\infty v_1^{-p'}(x)\biggl(\int_x^{a^{-1}(x)}|\chi_{[\eta_{-N},\eta_N]}g|\biggr)^{p'}dx.
\end{align*} 
The assertion follows from the fact that
\begin{equation*}
\int_{\eta_{-N-1}}^{\eta_{N}}v_1^{-p'}(x)\biggl(\int_x^{a^{-1}(x)}|\chi_{[\eta_{-N},\eta_N]}g|\biggr)^{p'}dx\le
\int_{\eta_{-N-1}}^{\eta_{N}}v_1^{-p'} \biggl(\int_{\eta_{-N-1}}^{\eta_{N+1}}|g|\biggr)^{p'}<\infty.
\end{equation*} 

\smallskip
Denote $G_{i,k}^{(\delta)}(t)=:H_{i,k}^{(\delta)}g(t)$, $i=1,2$.  We can write
\begin{align*}
\|g-g_N\|_{\mathscr{W}_{p',1/{v_1}}}^{p'}=&
\sum_{i=1,2}\sum_{\delta=1,2}\sum_{k\in\mathbb{Z}}\int_{\eta_{k-1}}^{\eta_k}
v_1^{-p'}(t)\bigl|H_{i,k}^{(\delta)}g(t)-H_{i,k}^{(\delta)}g_N(t)\bigr|^{p'}\,dt
\\=&\sum_{i=1,2}\sum_{\delta=1,2}\sum_{k\in\mathbb{Z}}\int_{\eta_{k-1}}^{\eta_k}
v_1^{-p'}(t)\bigl|H_{i,k}^{(\delta)}(\chi_{(0,\eta_{-N})}g)(t)+H_{i,k}^{(\delta)}(\chi_{(\eta_{N},\infty)}g)(t)\bigr|^{p'}\,dt
\\=&\sum_{i=1,2}\sum_{\delta=1,2}\sum_{k\le -N-1}\int_{\eta_{k-1}}^{\eta_k}
v_1^{-p'}\bigl|G_{i,k}^{(\delta)}\bigr|^{p'}+\sum_{\delta=1,2}\int_{\eta_{-N-1}}^{\eta_{-N}}
v_1^{-p'}\bigl|G_{1,N}^{(\delta)}\bigr|^{p'}\\&+
\sum_{\delta=1,2}\int_{\eta_{N-1}}^{\eta_{N}}
v_1^{-p'}\bigl|G_{2,N}^{(\delta)}\bigr|^{p'}+\sum_{i=1,2}\sum_{\delta=1,2}\sum_{k\ge N+1}\int_{\eta_{k-1}}^{\eta_k}
v_1^{-p'}\bigl|G_{i,k}^{(\delta)}\bigr|^{p'}\\ 
\le&\sum_{i=1,2}\sum_{\delta=1,2}\sum_{|k|\ge N}\int_{\eta_{k-1}}^{\eta_k}
v_1^{-p'}(t)\bigl|G_{i,k}^{(\delta)}(t)\bigr|^{p'}\,dt.
\end{align*} 
This approves the statement of Lemma in view of \eqref{r1}. 
\end{proof} 

Now we can make an addition to the last assertion of Theorem \ref{T1}.
\begin{rem}
Let $X=W_p^1$. Then, by Theorem \ref{T1},
$$X^\prime_w=\Bigl\{g\in \mathbb{W}_{p',1/{v_1}}: \|g\|_{X'_w}\approx\|g\|_{\mathscr{W}_{p',1/{v_1}}}<\infty\Bigr\}.$$
It follows from here that $X'_w\subseteq \mathscr{W}_{p',1/{v_1}}$, and the inclusion can be strict, since there are examples of $g_0\in\mathscr{W}_{p',1/{v_1}}$ when 
$g_0\not\in\mathbb{W}_{p',1/{v_1}}$ (see \cite[Remark 5.5]{PSU0}).

Indeed, if $g_0\in X'_w$ then, by \cite[Theorem 2.5]{PSU0}, 
$$\|g_0\|_{X'_w}=J_X(g_0)<\infty\quad \Longleftrightarrow \quad \infty>\mathbf{J}_X(g_0)=\|g_0\|_{\mathbb{W}_{p',1/{v_1}}}=\infty,$$ which is a contradiction.

Let \begin{multline*}X'_{\textrm{ext}}:=\bigl\{g\in\mathscr{W}_{p',1/{v_1}}\colon \textrm{ there exists } \{g_k\}\subset X'_w \textrm{ such that }  \\
\lim_{k\to\infty}\|g-g_k\|_{\mathscr{W}_{p',1/{v_1}}}=0 \textrm{ and } \|g\|_{X'_{\textrm{ext}}}:=\lim_{k\to\infty}\|g_k\|_{X'_w}\bigr\}.\end{multline*}
Notice that the definition of $X'_{\textrm{ext}}$ is independent of a choice of $\{g_k\}$. Then
$$
X'_{\textrm{ext}}\hookrightarrow\mathscr{W}_{p',1/{v_1}} \textrm { and } \|g\|_{\mathscr{W}_{p',1/{v_1}}}\le \|g\|_{X'_{\textrm{ext}}}.$$

Conversely, let $g\in\mathscr{W}_{p',1/{v_1}}$. Then, by Lemma \ref{plot}, there exists $\{g_k\}\subset \mathbb{W}_{p',1/{v_1}} \subset X'_w$ such that 
$\|g\|_{\mathscr{W}_{p',1/{v_1}}}=\lim_{k\to\infty}\|g_k\|_{\mathscr{W}_{p',1/{v_1}}}=\|g\|_{X'_{\textrm{ext}}}$. Hence,
$g\in X'_{\textrm{ext}}$ and we have $\mathscr{W}_{p',1/{v_1}}\subset X'_{\textrm{ext}}$ and $\|g\|_{X'_{\textrm{ext}}}=\|g\|_{\mathscr{W}_{p',1/{v_1}}}$. Thus, 
$$
X'_{\textrm{ext}}=\mathscr{W}_{p',1/{v_1}}
$$
with equality of the norms.
\end{rem}

The next technical statement is used in Corollary \ref{corol} to prove $[X'_w]'_s=\{0\}$. 

\begin{lem}\label{lemSW}
Let $1<p<\infty$, $[c,d]\subset (0,\infty)$ and $h\in L^1([c,d])$. Then for any $\varepsilon>0$ there exists $g\in \mathbb{W}_{p',1/{v_1}}$ such that
$|g|=|h|$ on $[c,d]$ and $\|g\|_{\mathscr{W}_{p',1/{v_1}}}<\varepsilon$.
\end{lem}

\begin{proof}
Firstly, we show that for $g$ with ${\rm supp}\,g\in[c,d]$ it holds
\begin{equation}\label{comp}
\|g\|_{\mathscr{W}_{p',1/{v_1}}}^{p'}\lesssim
[V_1(c)]^{p'+1}\biggl|\int_c^d \frac{g}{V_1}\biggr|^{p'}+
\int_{c}^{d}v_1^{-p'}(t)V_1^{p'}(t)\biggl|\int_t^d \frac{g}{V_1}\biggr|^{p'}\,dt.
\end{equation}
We start from the functional $\mathcal{G}(g)$, for which it holds, by the triangle inequality, that
\begin{align*}
\mathcal{G}(g\chi_{[c,d]})\le&\biggl(\int_{a(c)}^d v_1^{-p'}(t)\,V_1^{p'}(t)\, \biggl|\int_t^{a^{-1}(t)}\frac{\chi_{[c,d]}(x)g(x)}{V_1(x)}
\,dx\biggr|^{p'}\,dt\biggr)^{1/p'} \\\le& 
\biggl(\int_{a(c)}^d v_1^{-p'}(t)\,V_1^{p'}(t)\, \biggl|\int_t^d\frac{\chi_{[c,d]}(x)g(x)}{V_1(x)}
\,dx\biggr|^{p'}\,dt\biggr)^{1/p'} \\&+\biggl(\int_{a(c)}^{d} v_1^{-p'}(t)\,V_1^{p'}(t)\, \biggl|\int_{a^{-1}(t)}^d\frac{\chi_{[c,d]}(x)g(x)}{V_1(x)}
\,dx\biggr|^{p'}\,dt\biggr)^{1/p'}.
\end{align*} 
Since for any $\alpha>0$
\begin{equation}\label{cococo}
\int_{a(t)}^t v_1^{-p'}[V_1^+]^{\alpha}\le 
\int_{a(t)}^t v_1^{-p'}(x)\Bigl[\int_{a(t)}^{b(x)} v_1^{-p'}\Bigr]^{\alpha}\,dx\le [V_1(t)]^{\alpha+1}
,\end{equation} 
we have
\begin{align}\label{coco}&
\int_{a(c)}^d v_1^{-p'}(t)\,V_1^{p'}(t)\, \biggl|\int_t^d\frac{\chi_{[c,d]}(x)g(x)}{V_1(x)}
\,dx\biggr|^{p'}\,dt\nonumber\\&=
\int_{a(c)}^c v_1^{-p'}(t)\,V_1^{p'}(t)\, \biggl|\int_c^d\frac{g(x)}{V_1(x)}
\,dx\biggr|^{p'}\,dt+
\int_{c}^d v_1^{-p'}(t)\,V_1^{p'}(t)\, \biggl|\int_t^d\frac{g(x)}{V_1(x)}
\,dx\biggr|^{p'}\,dt\nonumber\\ &\lesssim 
[V_1(c)]^{p'+1}\biggl|\int_c^d \frac{g}{V_1}\biggr|^{p'}+
\int_{c}^{d}v_1^{-p'}(t)V_1^{p'}(t)\biggl|\int_t^d \frac{g}{V_1}\biggr|^{p'}\,dt.
\end{align}
By the substitution $y=a^{-1}(t)$ and in view of \eqref{eq} and $V_1^+(a(y))\le V_1(y)$,
\begin{multline*}
\int_{a(c)}^{d} v_1^{-p'}(t)\,V_1^{p'}(t)\, \biggl|\int_{a^{-1}(t)}^d\frac{\chi_{[c,d]}(x)g(x)}{V_1(x)}
\,dx\biggr|^{p'}\,dt=\int_{a(c)}^{a(d)} v_1^{-p'}(t)\,V_1^{p'}(t)\, \biggl|\int_{a^{-1}(t)}^d\frac{\chi_{[c,d]}(x)g(x)}{V_1(x)}
\,dx\biggr|^{p'}\,dt\\\le 
\int_{c}^{d} v_1^{-p'}(a(y))\,V_1^{p'}(a(y))a'(y)\, \biggl|\int_{y}^d\frac{\chi_{[c,d]}(x)g(x)}{V_1(x)}
\,dx\biggr|^{p'}\,dt\\\lesssim \int_{c}^{d} v_1^{-p'}(y)\,V_1^{p'}(y)\, \biggl|\int_{y}^d\frac{\chi_{[c,d]}(x)g(x)}{V_1(x)}
\,dx\biggr|^{p'}\,dt=
\int_{c}^{d} v_1^{-p'}(y)\,V_1^{p'}(y)\, \biggl|\int_{y}^d\frac{g(x)}{V_1(x)}
\,dx\biggr|^{p'}\,dt.
\end{multline*}
Therefore, the estimate \eqref{comp} for the component 
$\mathcal{G}(g\chi_{[c,d]})$ of $\|g\|_{\mathscr{W}_{p',1/{v_1}}}^{p'}$ now follows.
To prove the same for $\mathbb{G}(g\chi_{[c,d]})$ we write
\begin{align*}&
\int_0^\infty v_1^{-p'}(t)\biggl|\int_t^{a^{-1}(t)}\frac{\chi_{[c,d]}(x)g(x)}{V_1(x)}\biggl(\int_{a(x)}^tv_1^{-p'}\biggr)dx
\biggr|^{p'}\,dt\\&=
\int_{a(c)} ^dv_1^{-p'}(t)\biggl|\int_t^{a^{-1}(t)}\frac{\chi_{[c,d]}(x)g(x)}{V_1(x)}\biggl(\int_{a(x)}^tv_1^{-p'} \biggr)dx
\biggr|^{p'}\,dt
\\&=
\int_{a(c)} ^d v_1^{-p'}(t)\biggl|\int_{a(t)}^tv_1^{-p'} (y)\biggl(\int_t^{a^{-1}(y)}\frac{\chi_{[c,d]}(x)g(x)}{V_1(x)}\,dx\biggr)dy
\biggr|^{p'}\,dt.
\end{align*}
By the triangle and H\"{o}lder's inequalities,
\begin{align*}
\mathbb{G}(g\chi_{[c,d]})=&
\biggl(\int_{a(c)} ^d v_1^{-p'}(t)\biggl(\int_{a(t)}^tv_1^{-p'} (y)\biggl|\int_t^{a^{-1}(y)}\frac{\chi_{[c,d]}g}{V_1}\biggr|dy
\biggr)^{p'}\,dt\biggr)^{1/p'}\\\le&
\biggl(\int_{a(c)} ^d v_1^{-p'}(t)\biggl(\int_{a(t)}^tv_1^{-p'} (y)\biggl|\int_t^d\frac{\chi_{[c,d]}g}{V_1}\biggr|
 dy
\biggr)^{p'}\,dt\biggr)^{1/p'}\\&+
\biggl(\int_{a(c)} ^{a(d)} v_1^{-p'}(t)\biggl(\int_{a(t)}^tv_1^{-p'} (y)\biggl|\int_{a^{-1}(y)}^d\frac{\chi_{[c,d]}g}{V_1}\biggr|
 dy
\biggr)^{p'}\,dt\biggr)^{1/p'}\\\le&
\biggl(\int_{a(c)} ^d v_1^{-p'}(t)\,V_1^{p'}(t)\biggl|\int_t^d\frac{\chi_{[c,d]}g}{V_1}\biggr|
^{p'}\,dt\biggr)^{1/p'}\\&+
\biggl(\int_{a(c)} ^{a(d)} v_1^{-p'}(t)\,V_1^{p'-1}(t)\biggl(\int_{a(t)}^tv_1^{-p'} (y)\biggl|\int_{a^{-1}(y)}^d
\frac{\chi_{[c,d]}g}{V_1}\biggr|^{p'}
dy
\biggr)\,dt\biggr)^{1/p'}.
\end{align*} 
Since (see \eqref{cococo} and \eqref{eq})
\begin{align*}&
\int_{a(c)} ^{a(d)} v_1^{-p'}(t)\,V_1^{p'-1}(t)\biggl(\int_{a(t)}^tv_1^{-p'} (y)\biggl|\int_{a^{-1}(y)}^d
\frac{\chi_{[c,d]}(x)g(x)}{V_1(x)}\,dx\biggr|^{p'}
 dy
\biggr)\,dt\\&=
\int_{a(a(c))} ^{a(d)} v_1^{-p'} (y)\biggl|\int_{a^{-1}(y)}^d
\frac{\chi_{[c,d]}(x)g(x)}{V_1(x)}\,dx\biggr|^{p'}
\biggl(\int_y^{a^{-1}(y)} v_1^{-p'}(t)\,V_1^{p'-1}(t)
 dt\biggr)dy\\&\lesssim
\int_{a(c)} ^{a(d)} v_1^{-p'} (y)V_1^{p'}(a^{-1}(y))\biggl|\int_{a^{-1}(y)}^d
\frac{\chi_{[c,d]}(x)g(x)}{V_1(x)}\,dx\biggr|^{p'}\,dy\lesssim
\int_{c} ^{d} v_1^{-p'} (t)V_1^{p'}(t)\biggl|\int_{t}^d
\frac{g(x)}{V_1(x)}\,dx\biggr|^{p'}\,dy
\end{align*} 
the estimate \eqref{comp} for 
$\mathbb{G}(g\chi_{[c,d]})$ follows by taking into account \eqref{coco}.

\smallskip
Secondly, we fix $\varepsilon>0$ and take 
$$
n>\varepsilon^{-1}\Bigl(\int_c^d v_1^{-p'}\,V_1^{p'}\Bigr)^{1/p'}\int_c^d|h|.
$$ 
Let $\{\alpha_i\}_{i=0}^n$ be a partition of $[c,d]$ such that $\int_{\alpha_i}^{\alpha_{i+1}}|h|=n^{-1
}\int_c^d|h|$ and suppose $\beta_i\in[\alpha_i,\alpha_{i+1}]$ are such that $\int_{\alpha_i}^{\beta_{i}}|h|=\int_{\beta_i}^{\alpha_{i+1}}|h|$, $i\in\{0,\ldots,n-1\}$. Put
$$
\tilde{g}:=V_1|h|\sum_{i=0}^{n-1}\bigl(\chi_{[\alpha_i,\beta_i]}-
\chi_{(\beta_i,\alpha_{i+1)}}\bigr).
$$
Then $\tilde{g}\in \mathbb{W}_{p',1/{v_1}}$, $|\tilde{g}|=|h|$ on $[c,d]$, $\int_{\alpha_i}^{\alpha_{i+1}}\frac{\tilde{g}}{V_1}=0$ for $i=0,\ldots,n-1$ and (see \eqref{comp})
\begin{align*}
\|\tilde{g}\|_{\mathscr{W}_{p',1/{v_1}}}^{p'}\lesssim&
\int_{c}^{d}v_1^{-p'}(x)V_1^{p'}(x)\biggl|\int_x^d \frac{\tilde{g}}{V_1}\biggr|^{p'}\,dx
=\sum_{i=0}^{n-1}\int_{\alpha_i}^{\alpha_{i+1}}v_1^{-p'}(x)V_1^{p'}(x)\biggl|\int_x^{\alpha_{i+1}} \frac{\tilde{g}}{V_1}\biggr|^{p'}\,dx\\
=&
\sum_{i=0}^{n-1}\int_{\alpha_i}^{\alpha_{i+1}}v_1^{-p'}(x)V_1^{p'}(x)\biggl|\int_{\alpha_{i}}^x \frac{\tilde{g}}{V_1}\biggr|^{p'}\,dx
\le \sum_{i=0}^{n-1}\biggl(\int_{\alpha_{i}}^{\alpha_{i+1}} |h|\biggr)^{p'}\int_{\alpha_i}^{\alpha_{i+1}}v_1^{-p'}V_1^{p'}\\=&
n^{-p'}\biggl(\int_{c}^{d} |h|\biggr)^{p'}\sum_{i=0}^{n-1}\int_{\alpha_i}^{\alpha_{i+1}}v_1^{-p'}V_1^{p'}=
n^{-p'}\biggl(\int_{c}^{d} |h|\biggr)^{p'}\int_{c}^{d}v_1^{-p'}V_1^{p'}<\varepsilon^{p'}.
\end{align*}
\end{proof}

\begin{cor}\label{corol}
Let $f\in\mathfrak{M}(0,\infty)$. If ${\rm meas}\,\{x\in (0,\infty)\colon f(x)\not=0\}>0$ then $\mathbf{J}_{\mathscr{W}_{p',1/{v_1}}}(f)=\infty$.
\end{cor}
\begin{proof}
Let $f\not\equiv 0$. There is a segment $[c,d]\subset(0,\infty)$ such that $c<d$ and ${\rm meas}\,\bigl((c,d)\cap \{x\in (0,\infty)\colon f(x)\not=0\}\bigr)>0$. 
Fix an arbitrary $\varepsilon>0$. By Lemma \ref{lemSW} there exists $\tilde{g}\in\mathbb{W}_{p',1/{v_1}}$ with ${\rm supp}\,\tilde{g}\subset[c,d]$ such that 
$\|\tilde{g}\|_{\mathscr{W}_{p',1/{v_1}}}<\varepsilon$ and $|\tilde{g}|=1$ on $(c,d)$. Then 
$$
\mathbf{J}_{\mathscr{W}_{p',1/{v_1}}}(f)\ge\frac{\int_0^\infty|f\tilde{g}|}{\|\tilde{g}\|_{\mathscr{W}_{p',1/{v_1}}}}\ge \varepsilon^{-1}\int_c^d|f|.
$$
\end{proof}

\section{Main result}

We start with auxiliary assertions needed to prove the main result.

\begin{lem}\label{Em1} 
Let $1<p<\infty,$  $v_0,v_1\in {\mathscr V}_p(0,\infty)$, $\frac{1}{v_1}\in L^{p'}_\loc(0,\infty),$ and the condition \eqref{S6} is satisfied. Then
\beqn\label{N1}
L_{1/{v_0}}^{p'}(0,\infty)\subset \mathscr{W}_{p',1/{v_1}}
\eeqn 
and
\beqn\label{Norm1}
\|g\|_{\mathscr{W}_{p',1/{v_1}}}\lesssim \|g\|_{p',1/{v_0}}
\eeqn 
for any $g\in L_{1/{v_0}}^{p'}(0,\infty).$
\end{lem}
\begin{proof}
On the strength of 
\beqn\label{N2}
V_1^+(t)\le \int_{t}^{b(x)}v_1^{-p'}\le V_1(x)=2V_1^-(x), \qquad t\leq x\leq a^{-1}(t)
\eeqn
it holds
$$
\|g\|_{\mathscr{W}_{p',1/{v_1}}}\lesssim \biggl(\int_0^\infty v_1^{-p'}(t)\biggl(\int_t^{a^{-1}(t)}|g(x)|\,dx
\biggr)^{p'}\,dt\biggr)^{1/p'}.
$$
Then \eqref{Norm1} will follow from 
\beqn\label{N3}
\biggl(\int_0^\infty v_1^{-p'}(t)\biggl(\int_t^{a^{-1}(t)}|g(x)|\,dx
\biggr)^{p'}\,dt\biggr)^{1/p'}\le C\|g\|_{p',1/{v_0}}.
\eeqn
Consider the dual to \eqref{N3} inequality 
$$
\biggl(\int_0^\infty v_0^{p}(y)\biggl(\int_{a(y)}^y|f|
\biggr)^{p}\,dy\biggr)^{1/p}\le C\|f\|_{p,v_1},
$$ 
which is a consequence of
$$
\biggl(\int_0^\infty v_0^{p}(y)\biggl(\int_{a(y)}^{b(y)}|f|
\biggr)^{p}\,dy\biggr)^{1/p}\le C_1\|f\|_{p,v_1}.
$$ 
It is known \cite[Theorem 3.1]{PSU0} that 
$$
C_1\approx \mathcal{A}:=\sup_t\biggl(
\int_{a(t)}^{b(t)}v_1^{-p'}\biggr)^{1/p'}
\biggl(
\int_{b^{-1}(t)}^{a^{-1}(t)}v_0^{p}\biggr)^{1/p}.
$$ 
Put
\begin{equation*}
V_0(t):=\int_{a(t)}^{b(t)}v_0^p,\qquad V_0^\pm(t):=\int_{\Delta^\pm(t)}v_0^p.
\end{equation*}
We have by \eqref{N2}
$$
V_1^+(t)\le V_1(a^{-1}(t)),\qquad \int_t^{a^{-1}(t)}v_0^p\le 
\int_t^{b(a^{-1}(t))}v_0^p=:V_0^+(a^{-1}(t)).
$$ 
Therefore, by \eqref{3}, 
$$
\mathcal{A}_a(t):=
\biggl(
\int_{a(t)}^{b(t)}v_1^{-p'}\biggr)^{1/p'}
\biggl(
\int_{t}^{a^{-1}(t)}v_0^{p}\biggr)^{1/p}\le  
V_1(a^{-1}(t))^{1/p'}V_0(a^{-1}(t))^{1/p}=1.
$$ 
Analogously,
$$
\mathcal{A}_b(t):=
\biggl(
\int_{a(t)}^{b(t)}v_1^{-p'}\biggr)^{1/p'}
\biggl(
\int_{b^{-1}(t)}^tv_0^{p}\biggr)^{1/p}\le 
V_1(b^{-1}(t))^{1/p'}V_0(b^{-1}(t))^{1/p}=1.
$$ 
Thus, 
$$
\mathcal{A}\approx\sup_{t>0}[\mathcal{A}_a(t)+\mathcal{A}_b(t)]\lesssim 1
$$ 
and \eqref{Norm1} follows.
\end{proof}

\begin{cor}\label{Cor1} Let $1<p<\infty$ and $f\in \mathfrak{D}_{\mathscr{W}_{p',1/{v_1}}}$ {\rm(}see \eqref{D-X}{\rm)}. 
Under the conditions of Lemma \ref{Em1} the embedding \eqref{N1} entails
$f\in L_{v_0}^p(0,\infty)$ and
\begin{equation}\label{v0}
\infty>J_{\mathscr{W}_{p',1/{v_1}}}(f)\gtrsim \|f\|_{p,v_0}.
\end{equation}
\end{cor}
\begin{proof} By Lemma \ref{Em1}
\begin{align*}
J_{\mathscr{W}_{p',1/{v_1}}}(f)=\sup_{0\not=g\in \mathscr{W}_{p',1/{v_1}}} \frac{\Bigl|\int_0^\infty gf\Bigr|}{\|g\|_{\mathscr{W}_{p',1/{v_1}}}}
\gtrsim 
\sup_{0\not=g\in L_{1/{v_0}}^{p'}(0,\infty)} \frac{\Bigl|\int_0^\infty gf\Bigr|}{\|g\|_{p',1/{v_0}}}=\|f\|_{p,v_0}.
\end{align*}
\end{proof}

\begin{lem}\label{Em2} Let $1<p<\infty$.
Under the conditions of Lemma \ref{Em1}, if 
$J_{\mathscr{W}_{p',1/{v_1}}}(f)<\infty$ then $f=\tilde{f}$ a.e., where $\tilde{f}\in  AC_\loc(0,\infty)$ and 
\begin{equation}\label{dd}
\infty>J_{\mathscr{W}_{p',1/{v_1}}}(f)\gtrsim\|\tilde{f}'\|_{p,v_1}.
\end{equation}
\end{lem}
\begin{proof} Let
$$
g_\phi(x):=\frac{d\phi}{dx},\qquad \phi\in C_0^\infty(0,\infty).
$$ 
Show that $g_\phi\in \mathscr{W}_{p',1/{v_1}}$. It is sufficiently to prove  the inequalities
$\mathbb{G}(g_\phi)\lesssim \|\phi\|_{p',1/{v_1}}$ and $\mathcal{G}(g_\phi)\lesssim \|\phi\|_{p',1/{v_1}}$.
By taking into account the equality 
\begin{equation}\label{eq}
v_1^{-p'}(a(x))a'(x)+v_1^{-p'}(b(x))b'(x)= 2v_1^{-p'}(x),
\end{equation}
 which follows from the equilibrium condition \eqref{2}, we write
\begin{align}
\int_t^{a^{-1}(t)}\frac{g_\phi(x)}{V_1(x)}\biggl(\int_{a(x)}^tv_1^{-p'} \biggr)\ 
dx=-\phi(t)
+\int_t^{a^{-1}(t)}\phi(x)\biggl\{
\frac{v_1^{-p'}(a(x))a'(x)}{V_1^-(x)}\nonumber\\+
\frac{v_1^{-p'}(x)\int_{a(x)}^tv_1^{-p'}}{[V_1^-(x)]^2}-\frac{v_1^{-p'}(a(x))a'(x)\int_{a(x)}^tv_1^{-p'}}{[V_1^-(x)]^2}
\biggr\}\,dx 
\le |\phi(t)|+5 
\int_t^{a^{-1}(t)}\frac{v_1^{-p'}(x)|\phi(x)|}{V_1^-(x)}\,dx.
\end{align}
Thus,
$$
\mathbb{G}(g_\phi)\lesssim \|\phi\|_{p',v_1^{-1}} + \biggl(\int_0^\infty v_1^{-p'}(t)\biggl(\int_t^{a^{-1}(t)}\frac{v_1^{-p'}(x)|\phi(x)|}{V_1^-(x)}\,dx
\biggr)^{p'}\,dt\biggr)^{1/p'}.
$$
Put $h=v_1^{-1}|\phi|$ and consider dual to
\begin{equation}\label{02}
\biggl(\int_0^\infty v_1^{-p'}(t)\biggl(\int_t^{a^{-1}(t)}\frac{v_1^{1-p'}(x)h(x)}{V_1^-(x)}\,dx
\biggr)^{p'}\,dt\biggr)^{1/p'}\le C\|h\|_{p'}
\end{equation} 
inequality
$$
\biggl(\int_0^\infty \frac{v_1^{-p'}(x)}{[V_1^-(x)]^p}\biggl(\int_{a(x)}^x{v_1^{-1}(t)|\psi(t)|}\,dt
\biggr)^{p}\,dx\biggr)^{1/p}\le C\|\psi\|_{p},
$$ 
which follows from
$$
\biggl(\int_0^\infty \frac{v_1^{-p'}(x)}{[V_1^-(x)]^p}\biggl(\int_{a(x)}^{b(x)}{v_1^{-1}(t)|\psi(t)|}\,dt
\biggr)^{p}\,dx\biggr)^{1/p}\le C_2\|\psi\|_{p}.
$$ 
By the criteria for the boundedness of Hardy--Steklov operators \cite[Theorem 1]{SU},
$$
C_2\approx\mathscr{A}:=\sup_t\biggl(\int_{a(t)}^{b(t)}v^{-p'}_1\biggr)^{1/p'}\biggl(\int_{b^{-1}(t)}^{a^{-1}(t)}
\frac{v^{-p'}_1}{[V_1^-]^p}\biggr)^{1/p}.
$$
Since 
$\int_{b^{-1}(t)}^{a^{-1}(t)}
{v^{-p'}_1}{[V_1^-]^{-p}}\lesssim V_1^{1-p}(t)$ (see \cite[(5.18)]{PSU1}) we have $\mathscr{A}\lesssim 1$. This yields 
$\mathbb{G}(g_\phi)\lesssim \|\phi\|_{p',v_1^{-1}}<\infty$.

Similarly,
\begin{multline*}
\int_t^{a^{-1}(t)}\frac{g_\phi(x)}{V_1(x)} 
dx=\frac{\phi(a^{-1}(t))} {V_1^-(a^{-1}(t))}-\frac{\phi(t)}{V_1^-(t)}+
\int_t^{a^{-1}(t)}\phi(x)
\frac{v_1^{-p'}(x)-v_1^{-p'}(a(x))a'(x)}{[V_1^-(x)]^2}\,dx\\\le \frac{|\phi(a^{-1}(t))|} {V_1^-(a^{-1}(t))}+\frac{|\phi(t)|}{V_1^-(t)}+
\int_t^{a^{-1}(t)}|\phi(x)|
\frac{v_1^{-p'}(x)+ v_1^{-p'}(a(x))a'(x)}{[V_1^-(x)]^2}\,dx.
\end{multline*}
Since $2v_1^{-p'}(a^{-1}(t))[a^{-1}(t)]'\ge
v_1^{-p'}(t)$ (see \eqref{eq} with $x=a^{-1}(t)$) and 
$V_1(a^{-1}(t))\ge V_1^+(t)=\frac{1}{2}V_1(t)$, we have
\begin{align*}&
\int_0^\infty v_1^{-p'}(t)\,V_1^{p'}(t)\biggl[\frac{|\phi(a^{-1}(t))|} {V_1^-(a^{-1}(t))}+\frac{|\phi(t)|}{V_1^-(t)}\biggr]^{p'}\,dt\\
\lesssim &\int_0^\infty [a^{-1}(t)]'|\phi(a^{-1}(t))v_1^{-1}(a^{-1}(t))|^{p'}\,dt
+\int_0^\infty |\phi(t)v_1^{-1}(t)|^{p'}\,dt\simeq \|\phi\|_{p',v_1^{-1}}.
\end{align*} 
Further,
\begin{align*}&
\int_0^\infty v_1^{-p'}(t)\,V_1^{p'}(t)\biggl(\int_t^{a^{-1}(t)}|\phi(x)|
\frac{v_1^{-p'}(x)+ v_1^{-p'}(a(x))a'(x)}{[V_1^-(x)]^2}\,dx\biggr)^{p'}\,dt\\ &\le 
\int_0^\infty v_1^{-p'}(t)\biggl(\int_t^{a^{-1}(t)}|\phi(x)|
\frac{v_1^{-p'}(x)+ v_1^{-p'}(a(x))a'(x)}{V_1^-(x)}\,dx\biggr)^{p'}\,dt\\ &\le3 
\int_0^\infty v_1^{-p'}(t)\biggl(\int_t^{a^{-1}(t)}\frac{v_1^{-p'}(x)|\phi(x)|}{V_1^-(x)}\,
dx\biggr)^{p'}\,dt
\simeq\int_0^\infty v_1^{-p'}(t)\biggl(\int_t^{a^{-1}(t)}\frac{v_1^{1-p'}(x)h(x)}{V_1^-(x)}\,dx\biggr)^{p'}\,dt
\end{align*} 
(see \eqref{02}). Therefore, $\mathcal{G}(g_\phi)\lesssim \|\phi\|_{p',1/{v_1}}<\infty$ and 
\begin{equation}\label{D1}
\|g_\phi\|_{\mathscr{W}_{p',1/{v_1}}}\lesssim \|\phi\|_{p',1/{v_1}}.
\end{equation} 
It follows from \eqref{D1} that
\begin{align}\label{dds}
\sup_{0\not=\phi\in C^\infty_0(0,\infty)}\frac{\Bigl|\int_0^\infty f\phi'\Bigr|}{\|\phi\|_{p',1/{v_1}}}&\lesssim \sup_{0\not=\phi\in C^\infty_0(0,\infty)}
\frac{\Bigl|\int_0^\infty fg_\phi\Bigr|}{\|g_\phi\|_{\mathscr{W}_{p',1/{v_1}}}}\nonumber\\
&\le\sup_{g\in \mathscr{W}_{p',1/{v_1}}}\frac{\Bigl|\int_0^\infty fg\Bigr|}{\|g\|_{\mathscr{W}_{p',1/{v_1}}}}=J_{\mathscr{W}_{p',1/{v_1}}}(f)<\infty.
\end{align}
Put $\Lambda\phi:=\int_0^\infty f\phi'$, $\phi\in C^\infty_0(0,\infty)$. On the strength of \eqref{dds}, $|\Lambda\phi|\lesssim\|\phi\|_{p',1/{v_1}}$. 
By the Hahn--Banach theorem there exists an extension $\tilde{\Lambda}\in \bigl(L^{p'}_{1/{v_1}}(0,\infty)\bigr)^\ast$ of $\Lambda$. By the Riesz representation theorem, 
there exists $u\in L^p_{v_1}(0,\infty)$ such that $\tilde{\Lambda}h=-\int_0^\infty uh$, $h\in L^{p'}_{1/{v_1}}(0,\infty)$. It implies
\begin{equation}\label{d1}
-\int_0^\infty u\phi=\int_0^\infty f\phi',\qquad\phi\in C_0^\infty(0,\infty),
\end{equation} 
this means that $u$ is a distributional derivative of $f$. Then by \cite[Theorem 7.13]{Leoni} the function $f$ a.e. coincides with a function $\tilde f\in AC_\loc(0,\infty)$ and $u=\tilde{f}^\prime.$
It follows from \eqref{d1} that
\begin{align*}
J_{\mathscr{W}_{p',1/{v_1}}}(f)\ge\sup_{0\not=\phi\in C^\infty_0(0,\infty)}\frac{\Bigl|\int_0^\infty fg_\phi\Bigr|}{\|g_\phi\|_{\mathscr{W}_{p',1/{v_1}}}}
\gtrsim\sup_{0\not=\phi\in C^\infty_0(0,\infty)}\frac{\Bigl|\int_0^\infty \tilde{f}'\phi\Bigr|}{\|\phi\|_{p',1/{v_1}}}=
\|\tilde{f}'\|_{p,v_1}.
\end{align*}
\end{proof}

Our main result of the paper reads the following

\begin{thm}\label{theoremMain}
Let $1<p<\infty$ and $f\in\mathfrak{D}_{\mathscr{W}_{p',1/{v_1}}}$. Then $J_{\mathscr{W}_{p',1/{v_1}}}(f)<\infty$ if and only if $f=\tilde{f}$ a.e., 
$\tilde{f}\in \wcnn^1_p$ and $\|f\|_{W^1_p}\approx J_{\mathscr{W}_{p',1/{v_1}}}(f).$ Thus,
$$
\wcnn^1_p=[\mathscr{W}_{p',1/{v_1}}]_w^\prime=[[\wcnn^1_p]^\prime_w]_w^\prime.
$$
\end{thm}
\begin{proof} The {\it sufficient} part of Theorem follows from Remark \ref{remark}. 
 
{\it Necessity}. It is already proved that $\tilde{f}\in W^1_p$ (see \eqref{v0}, \eqref{dd}). Let $E:=\{x\in (0,\infty):|\tilde f(x)|>0\}$. Since $|\tilde f|$ is continuous on $(0,\infty)$, then 
$E$ is open set.  Suppose that  ${\rm mes}((b,\infty)\cap E)>0$ for any $b\in (0,\infty)$.  Then there exists a sequence of segments $\{[a_k,b_k]\}_1^\infty\subset (0,\infty)$ such that 
$b_k<a_{k+1}$ and $m_k:=\min_{x\in [a_k,b_k]}|\tilde f(x)|>0$. 
Put $\theta_k:=\frac{1}{k m_k(b_k-a_k)}$. 
By Lemma  \ref{lemSW}  there is $g_k\in \mathscr{W}_{p',1/{v_1}}$ such that $\|g_k\|_{\mathscr{W}_{p',1/{v_1}}}<2^{-k}$ and $|g_k|=\theta_k$ on $(a_k,b_k)$.
Put 
 $g:=\sum_{k=1}^\infty g_k$.
 Then $\|g\|_{\mathscr{W}_{p',1/{v_1}}}\le 1$
  and
  $$
  \int_0^\infty|\tilde{f}g|\ge \sum_{k=1}^\infty \theta_k m_k(b_k-a_k)=\sum_{k=1}^\infty\frac{1}{k}=\infty,
    $$
  which contradicts with $J_{\mathscr{W}_{p',1/{v_1}}}(f)<\infty.$

Similarly we show that $\tilde{f}(0)=0.$ Thus, $\supp\tilde{f}\subset [0,\infty)$ is compact. 
\end{proof}




\end{document}